\documentclass[11pt,a4paper,twoside,reqno]{amsart}
\usepackage{amsfonts,amssymb,euscript,amscd}
\DeclareFontFamily{OML}{cyr}{} \DeclareFontShape{OML}{cyr}{m}{n}{
  <5> <6> <7> <8> <9> gen * wncyr
  <10> <10.95> <12> <14.4> <17.28> <20.74> <24.88> wncyr10
  }{}
\DeclareSymbolFont{rusletters}{OML}{cyr}{m}{n}
\DeclareSymbolFontAlphabet{\rusmath}{rusletters}
\DeclareMathSymbol\re{\rusmath}{rusletters}{"03}
\newtheorem{theorem}{Theorem}[section]
\newtheorem{proposition}[theorem]{Proposition}
\newtheorem{lemma}[theorem]{Lemma}

\newcommand*{\im}{\mathop{\rm Im}\nolimits}
\newcommand*{\id}{\mathop{\rm id}\nolimits} % Identical map
\newcommand*{\A}{\EuScript A} % Isotropy space
\newcommand*{\K}{\EuScript K} % Horisontal space

\title[On the obstruction to integrability]{On the obstruction to integrability of almost-complex structures}
\author[V.A. Yumaguzhin]{Valeriy A. Yumaguzhin}
\date{11 January 2008}
\address{Program Systems Institute of RAS, Pereslavl'-Zalesskiy, 152020, Russia}
\email{yuma@diffiety.botik.ru}
\begin{document}
\begin{abstract} 
The natural bundle $\pi:E\to M$ of almost-complex structures is considered. 
The action of the pseudogroup of all diffeomorphisms of $M$ on the total space
$E$ is investigated. A nontrivial 1-st order differential   invariant of this 
action is constructed. It is proved that the Nijenhuise tensor of an almost-complex 
structures is equal to zero iff the constructed invariant for this structure is zero.
\end{abstract}
\maketitle
%
%%%%%%%%%%%%%%%%%%%%%%%%%%%%%%%%%%%%%%%%%%%%%%%%%%%%%%%%%%%%%%%%%%%%%%%%%%%%%%%%%%%%%%%%%%%%%%%%%%%%%%%%%%%%%%%%%%%%%%%
\section{Introduction}
It is well known that structure functions of $G$--structures and their prolongations are differential invariants of corresponding geometric structures, see \cite{Strnbrg}. However it is inconvenient to use these invariants in some cases to investigate the equivalence problem of geometric structures. Very often it is more suitable to make use of differential invariants defined in jet bundles of the natural bundle of a considering geometric structure. Therefore it seems natural to construct differential invariants like structure functions of $G$--structures directly on jet bundles of natural bundles of geometric structures. We do it by example of the natural bundle of ordinary differential equations $y''=a^3(x,y)y'\,^3+a^2(x,y)y'\,^2+a^1(x,y)y'+a^0(x,y)$ in \cite{Yum1} and \cite{Yum2}. In this paper, we construct a differential invariant like the structure function of $G$--structure directly on the 1-jet bundle of the natural bundle of almost--complex structures. 

By $\pi:E\to M$ we denote the natural bundle of all almost-complex structures on the manifold $M$ and by $\Gamma$ we denote the pseudogroup of all diffeomorphisms of $M$. Every diffeomorphism $f\in\Gamma$ is lifted in the natural way to the diffeomorphism $f^{(k)}$ of the bundle $J^k\pi$ of all $k$--jets of all sections of $\pi$, $k=0,1,2,\ldots\;$. Thus the pseudogroup $\Gamma$ acts by its lifted diffeomorphisms on every $J^k\pi$. Invariants of this action are  differential invariants (of order $k$) of the action of $\Gamma$ on $\pi$. 

The constructed invariant is a differential invariant of 1-st order of this action. It is the function $\chi$ on $J^1\pi$ with values in Spencer cohomology groups:
$$
  \chi: j^1_pS\longmapsto\chi(j^1_pS)\in H^{0,2}(g_{S(p)})\quad\forall\,j^1_pS\in J^1\pi \,,
$$
where $S$ is a section of $\pi$, $j^1_pS$ is its 1--jet at $p\in M$, and $H^{0,2}(g_{S(p)})$ is the Spencer cohomology group generated by the isotropy algebra $g_{S(p)}$ of the almost--complex structure $S(p)$ on the tangent spase $T_pM$ to $M$ at $p$.

Further, we choose the natural representative $\omega_{j^1_pS}$ in every class $\chi(j^1_pS)$. As a result, we obtain the differential invariant 
$$
  \omega:j^1_pS\longmapsto \omega_{j^1_pS}\in T_pM\otimes(\wedge^2T_p^*M)\,.
$$

Every section $S$ of $\pi$ generates the section $j_1S:p\mapsto j_p^1S\,,\,\forall\,p\in M$ of the bundle $J^1\pi$.
The restriction $\omega |_{j_1S}$ of $\omega$ to the image of $j_1S$ can be considered as an invariant vector--valued 2--form on $M$. We prove that the Nijenhuise tensor of the almost-complex structures $S$ is equal to zero iff the $\omega |_{j_1S}$ is zero. 
\smallskip

All manifolds and maps are smooth in this work. By $j_p^kf$ denote the $k$--jet of the map $f$ at the point $p$, $k=0,1,2,\ldots$. We assume summation over repeated indexes in all formulas.

%%%%%%%%%%%%%%%%%%%%%%%%%%%%%%%%%%%%%%%%%%%%%%%%%%%%%%%%%%%%%%%%%%%%%%%%%%%%%%%%%%%%%%%%%%%%%%%%%%%%%%%%%%%%%%%%%%%%%%%
\section{The natural bundle of almost-complex structures}
%----------------------------------------------------------------------------------------------------------------------
\subsection{The bundle of tensors of type $(1,1)$}
Let $M$ be a $2n$ -- dimensional smooth manifold. Consider the natural 
bundle 
$$
  \tau:TM\otimes T^*M\longrightarrow M
$$ 
of all tensors of type $(1,\,1)$ over $M$. By 
$$
  \tau_k:J^k\tau\longrightarrow M
$$ 
denote the bundle of $k$--jets of sections of $\tau$, $k=1,2,\ldots$\, A local chart $\bigl(U,\,(x^1,\ldots,x^{2n})\bigr)$ of $M$ generates in the obvious way the local chart in $J^k\tau$
$\bigl(\tau^{-1}_k(U),\,(x^i,\,u^i_j,\,u^i_{j,j_1},\,\ldots,\,u^i_{j,j_1\ldots j_k})\bigr)$, where $i,j,j_1,\ldots,j_k=1,\ldots,2n$. We say that the coordinates of this chart are {\it standard coordinates of} $J^k\tau$.

Let $f$ be a diffeomorphism of $M$. It is naturally lifted to the diffeomorphism $f^{(0)}$ of the total space of $\tau$. In a standard coordinates, $f^{(0)}$ is described in the following way. Suppose $f$ is described in coordinates of $M$ by the equations  
$$
  \tilde x^i=f^i(x^1,\ldots,x^{2n})\,.
$$ 
Then $f^{(0)}$ is described in the standard coordinates by the formulas
\begin{equation}\label{Lft0D}
 \begin{aligned}
  \tilde x^i&=f^i(x^1,\ldots,x^{2n})\,,\\
  \tilde u^i_j&=f^i_r(x)\,u^r_s\,g^s_j\bigl(f(x)\bigr)\,,
 \end{aligned} 
\end{equation}
where $x=(x^1,\ldots,x^{2n})$, $f^i_r=\partial f^i/\partial x^r$, $g=f^{-1}$, and $g^s_j=\partial g^i/\partial\tilde x^j$. It is clear that the following diagram is commutative 
$$
  \begin{CD}
    TM\otimes T^*M @>f^{(0)}>> TM\otimes T^*M\\
    @V\tau VV      @VV\tau V\\
    M       @>>f>  M\,.
  \end{CD}
$$

Every diffeomorphism $f$ of $M$ is lifted in the natural way to the diffeomorphism $f^{(k)}$ of $J^k\tau$. This lifted diffeomorphism is defined by the formula
$$
  f^{(k)}(\,j^k_pS\,)=j^k_{f(p)}\bigl(\,f^{(0)}\circ S\circ f^{-1}\,\bigr)\,,
$$
where $S$ is a section of $\tau$, $j^k_pS$ is its $k$--jet at the point $p\in M$. 

The natural projection 
$$
  \tau_{l,m}:J^l\tau\longrightarrow J^m\tau\,,\quad l>m\geq 0\,,
$$ 
is defined by $\tau_{l,m}(j_p^lS)=j^m_pS$, $\forall\,j_p^lS\in J^l\pi$.
Obviously, the diagram
$$
  \begin{CD}
    J^l\tau         @>f^{(l)}>> J^l\tau\\
    @V\tau_{l,\;m}VV             @VV\tau_{l,\;m} V\\
    J^m\tau         @>>f^{(m)}> J^m\tau
  \end{CD}
$$
is commutative. 

Suppose $f$ and $g$ are diffeomorphisms of $M$, then obviously,
$$
  (f\circ g)^{(k)}=f^{(k)}\circ g^{(k)},\;\;k=0,1,\ldots
$$

Let $X$ be a vector field in $M$ and let $f_t$ be its flow. Then the lifted flow $f_t^{(0)}$ defines the vector field $X^{(0)}$ in the total space of $\tau$, which is called {\it the lifting of $X$ to the total space of $\tau$}. It is clear that
$$
 \tau_*\bigl(\,X^{(0)}\,\bigr)=X\,.
$$
Suppose $X$ is described in coordinates of $M$ by the equation
$$
  X=X^1(x)\frac{\partial}{\partial x^1}+\ldots +X^{2n}(x)\frac{\partial}{\partial x^{2n}}\,.
$$
Then it follows from \eqref{Lft0D} that $X^{(0)}$ is described in the standard coordinates $x^i, u^i_j$ by the formula
\begin{equation}\label{Lft0VF}
  X^{(0)}=X^i(x)\frac{\partial}{\partial x^i}+\bigl(X^i_s(x)\,u^s_j-u^i_s\,X^s_j(x)\bigr)\frac{\partial}{\partial u^i_j}\,,
\end{equation}
where $X^i_j=\partial X^i/\partial x^j$.

%----------------------------------------------------------------------------------------------------------------------
\subsection{The bundle of almost-complex structures}
Recall that a section $S$ of $\tau$ is called an almost-complex structures on $M$ if it satisfies to the equation 
\begin{equation}\label{AlmCmplStr}
  S(p)^2=-\id_{\,T_pM},\quad\forall\,p\in M
\end{equation}
where $\id_{\,T_pM}$ is the identity map of the tangent space $T_pM$ to itself. The equation \eqref{AlmCmplStr} defines the subbundle 
$$
  \pi=\tau|_E:E\longrightarrow M
$$ 
of the bundle $\tau$, where 
$$
  E=\{\,\theta\in TM\otimes T^*M\,|\,\theta^2=-\id_{\,T_pM},\; p=\tau(\theta)\,\}\,.
$$
It is clear that the set of all section of $\pi$ is the set of all almost-complex structures on $M$.

It is easy to check that for every diffeomorphism $f$ of $M$, its lifted diffeomorphism $f^{(0)}$ transforms the total space $E$ of $\pi$ to itself. Hence every diffeomorphism $f$ of $M$ is naturally lifted to diffeomorphism $f^{(0)}|_E$ of the total space of $\pi$. Thus $\pi$ is a natural bundle. 

It is clear that for every vector field $X$ in $M$, its lifted vector field $X^{(0)}$ is tangent to the submanifold $E$.

By 
$$
  \pi_k:J^k\pi\longrightarrow M
$$ 
denote the bundle of $k$--jets of sections of $\pi$, $k=1,2,\ldots$\,. The natural projection $\pi_{1,0}:J^1\pi\to E$ is defined by $\pi_{1,0}(j_p^1S)=S(p)$, $\forall\,j_p^1S\in J^1\pi$.

It follows from \eqref{AlmCmplStr} that $J^1\pi$ 
is described as a submanifold of $J^1\tau$ in standard coordinates by the equations
\begin{equation}\label{J1pi}
 \begin{aligned}
  &u^i_ru^r_j=-\delta^i_j\,,\\
  &u^i_{r,k}u^r_j+u^i_ru^r_{j,k}=0\,,
 \end{aligned} 
\end{equation}
where $\delta^i_j$ is the Kronecker delta.

By $\Gamma$ we denote the pseudogroup of all diffeomorphisms of $M$. 

It is clear that for every $f\in\Gamma$, its lifted diffeomorphism $f^{(k)}$ transforms submanifold $J^k\pi\subset J^k\tau$ to itself. 
We will denote the restriction $f^{(k)}|_{J^k\pi}$ by $f^{(k)}$, $k=0,1,2,\ldots\;$. 

Thus the pseudogroup $\Gamma$ acts on every $J^k\pi$ by the lifted transformations. Invariants of this action are differential invariants (of order $k$) of the action of $\Gamma$ on $\pi$. 

%%%%%%%%%%%%%%%%%%%%%%%%%%%%%%%%%%%%%%%%%%%%%%%%%%%%%%%%%%%%%%%%%%%%%%%%%%%%%%%%%%%%%%%%%%%%%%%%%%%%%%%%%%%%%%%%%%%%%%%%
\section{Isotropy algebras and spaces ${\bf \A_{\theta_1} }$}
%
%-----------------------------------------------------------------------------------------------------------------------
\subsection{Isotropy algebras}
Let $X$ be a vector field in $M$ and let $p$ be a point of the domain of $X$. Then it follows from \eqref{Lft0VF} that for every $\theta_0\in\pi^{-1}(p)$, the value $X^{(0)}_{\theta_0}$ of the lifted vector field $X^{(0)}$ at $\theta_0$ is defined by the 1--jet $j_p^1X$ of $X$ at $p$.

Let $\theta_0\in J^0\pi$ and $p=\pi(\theta_0)$. Consider all vector fields $X$ on $M$ passing through $p$. The isotropy algebra of the point $\theta_0$ is defined by the formula 
\begin{equation}\label{IztrpAlg0}
  g_{\theta_0}=\bigl\{\,j^1_pX\,\bigr|\,X^{(0)}_{\theta_0}=0\,\bigr\}\,.
\end{equation}
From this definition, we have that if $j^1_pX\in g_{\theta_0}$, then $X_p=0$.
Suppose $\theta_0=(x^i,\, u^i_j)$ in standard coordinates and $j^1_pX=(X^i,\, X^i_j)$ in the coordinates $x^1,\,\ldots,\,x^{2n}$ in $M$. Then it follows from \eqref{Lft0VF} that $(X^i,\, X^i_j)\in g_{\theta_0}$ iff the components $X^q$ and $X^i_j$ satisfy the equations
\begin{align}
  &X^i=0\,,\notag\\
  &X^i_ru^r_j-u^i_rX^r_j=0\label{IstrAlg0}\,.
\end{align}
Let us move $\theta_0$ by an appropriate lifted diffeomorphism $f^{(0)}$ to $\tilde\theta_0\in\pi^{-1}(p)$ such that $\bigl(u^i_j(\tilde\theta_0)\bigr)=\begin{pmatrix}0&-I\\I& 0\end{pmatrix}$, where $I$ is the identity $n\times n$--matrix. From \eqref{IstrAlg0}, we get that the algebra $g_{\tilde\theta_0}$ consists of all matrix of the form
$\begin{pmatrix}A & B\\-B & A\end{pmatrix}$, where $A$ and $B$ are an arbitrary $n\times n$--matrixes. Hence $\dim g_{\tilde\theta_0}=2n^2$. It implies that
$$
  \dim g_{\theta_0}=2n^2\,.
$$ 
The algebra $g_{\theta_0}$ can be considered as a subspace of $T_pM\otimes T_p^*M$. Indeed, let $\xi: TM\to M$ be the tangent bundle of $M$, $\xi_1: J^1\xi\to M$ its bundle of 1--jets of sections of $\xi$, and $\xi_{1,0}:J^1\xi\to TM$ the natural projection defined by $\xi_{1,0}:j_p^1X\mapsto X_p$, $\forall\,j_p^1X\in J^1\xi$. Consider the well known exact sequence
$$
  0\xrightarrow{} T_pM\otimes T_p^*M\xrightarrow{\mu} J^1_p\xi\xrightarrow{\xi_{1,0}} T_pM\xrightarrow{} 0\,,
$$
where $\mu$ is the linear map defined for decomposable elements by the formula $\mu(X_p\otimes d\varphi)=j^1_p(\varphi X)$, $X$ is a vector field on $M$, and $\varphi$ is a smooth function on $M$ such that $\varphi(p)=0$. From this sequence, we get that $g_{\theta_0}\subset \ker\xi_{1,0}\cong T_pM\otimes T_p^*M$. It implies $g_{\theta_0}\subset T_pM\otimes T_p^*M$. 

Recall that the subspase of $T_pM\otimes(T_p^*M\odot T_p^*M)$ defined by
$$
  g_{\theta_0}^{(1)}=(g_{\theta_0}\otimes T_p^*M)\cap\bigl(T_pM\otimes(T_p^*M\odot T_p^*M)\bigr)
$$
is called the 1-st prolongation of $g_{\theta_0}$.
The following Spencer complex is connected in the natural way with the algebra $g_{\theta_0}$
\begin{equation}\label{SpncrCmplx}
  0\xrightarrow{} g_{\theta_0}^{(1)}\hookrightarrow g_{\theta_0}\otimes T_p^*M 
  \xrightarrow{\partial_{1,1}}T_pM\otimes(\wedge^2 T_p^*M)\xrightarrow{} 0\,,
\end{equation}
here $\partial_{1,1}(h)(X_p,Y_p)=h(X_p)(Y_p)-h(Y_p)(X_p)$, $\forall\,h\in g_{\theta_0}\otimes T_p^*M$, $\forall\,X_p,Y_p\in T_pM$. By $H^{0,2}(g_{\theta_0})$ we denote the cohomology group of this complex at the term $T_pM\otimes\wedge^2 T_p^*M$. 
\begin{proposition} 
 \begin{enumerate}
  \item If $n=1$, then $H^{0,2}(g_{\theta_0})$ is trivial.
  \item If $n\geq 2$, then $H^{0,2}(g_{\theta_0})$ is not trivial.
 \end{enumerate} 
\end{proposition}  
\begin{proof} (1) It is easy to calculate that $\dim g_{\theta_0}^{(1)}=2$. Taking into account that $\dim g_{\theta_0}=2$, from \eqref{SpncrCmplx}, we get that $H^{0,2}(g_{\theta_0})=\{0\}$.

(2) It is clear that the number of linear independent equations in system \eqref{IstrAlg0} is $4n^2-2n^2=2n^2$. It follows that the number of linear independent equations in the system defining $g_{\theta_0}^{(1)}$ less or equal $4n^2$. The number of unknowns in this system is $4n^3+2n^2$. Thus $\dim g_{\theta_0}^{(1)}\geq 4n^3+2n^2-4n^2=4n^3-2n^2$. Taking into account that $\dim g_{\theta_0}\otimes T_p^*M=4n^3$, $\dim T_pM\otimes(\wedge^2 T_p^*M)=4n^3-2n^2$, from \eqref{SpncrCmplx}, we obtain that $\dim\im\partial_{1,1}\leq 2n^2<4n^3-2n^2$. This means that the cohomology group $H^{0,2}(g_{\theta_0})$ is not trivial.
\end{proof}

%-----------------------------------------------------------------------------------------------------------------------
\subsection{Spaces ${\bf \A_{\theta_1}}$}
Let $\theta_1\in J^1\pi$, $p=\pi_1(\theta_1)$, and $S$ be a section of $\pi$ such that $j^1_pS=\theta_1$. 
By $\K_{\theta_1}$ we denote the tangent space to the image of the section $S$ at the point $\theta_0=S(p)$. Obviously, 
this space is independent of the choice of a section of $\pi$ realizing the jet $\theta_1$. This means that $\theta_1$ is identified in the natural way with the horizontal subspace $\K_{\theta_1}$. In standard coordinates, $\K_{\theta_1}$ is described in the following way. Suppose $\theta_1=(\,x^q,\, u^i_j,\, u^i_{j,k}\,)$. Then
$$
  \K_{\theta_1}=\langle\;\frac{\partial}{\partial x^1}+u^i_{j,1}\frac{\partial}{\partial u^i_j},\;\ldots,\;
  \frac{\partial}{\partial x^{2n}}+u^i_{j,2n}\frac{\partial}{\partial u^i_j}\;\rangle.
$$
Let $f$ be a (local) diffeomorphism of $M$. Then it is obvious that if $\theta_1$ belongs to the domain of $f^{(1)}$, then 
\begin{equation}\label{TrnsfrmK1}
  f^{(0)}_*(\K_{\theta_1})=\K_{f^{(1)}(\theta_1)}\,.
\end{equation}
Let $V_{\theta_0}$ be the tangent space to the fiber $\tau^{-1}(p)$ at the point $\theta_0$. Clearly that in standard coordinates, $V_{\theta_0}$ is spanned by all vectors $\partial /\partial u^i_j$, that is
$$
  V_{\theta_0}=\langle\;\frac{\partial}{\partial u^1_1},\;\ldots,\;\frac{\partial}{\partial u^{2n}_{2n}}\;\rangle.
$$   
Obviously, we have the following direct sum decomposition
$$
  T_{\theta_0}J^0\tau=\K_{\theta_1}\oplus V_{\theta_0}\,.
$$
Consider all vector fields $X$ in $M$ passing through the point $p$. Introduce a subspace $\A_{\theta_1}$ of vector space of 1--jets at $p$ of these vector fields by the formula
\begin{equation}\label{IstrpSpc}
  \A_{\theta_1}=\bigl\{\;j_p^1X\;\bigr|\;X^{(0)}_{\theta_0}\in\K_{\theta_1}\;\bigr\}\,.
\end{equation}
From \eqref{Lft0VF}, we get that $(X^q,\, X^i_j)\in\A_{\theta_1}$ iff the components $X^q$ and $X^i_j$ satisfy the equations
\begin{equation}\label{EqIstrpSpc}
  -u^i_{j,r}\,X^r+ X^i_r\,u^r_j-u^i_r\,X^r_j=0\,,
\end{equation}
where the coordinates $u^i_j$ and $u^i_{j,r}$ satisfy equations \eqref{J1pi}.

Let $f$ be a diffeomorphism of $M$ and let $p$ be a point of the domain of $f$. The tangent map 
$f_*:T_pM\to T_{f(p)}M$ generates the map 
$$
  j_p^2f:J_p^1\xi\longrightarrow J_{f(p)}^1\xi\,,\quad
  j_p^2f:j_p^1X\mapsto j_{f(p)}^1\bigl(f_*(X)\bigr)\,.
$$
\begin{proposition}\label{PrpstnTrnsfrmIstrpSpc} Let $\theta_1$ be a point of the domain of $f^{(1)}$. Then
\begin{equation}\label{TrnsfrmIstrpSpc} 
 j_p^2f(\A_{\theta_1})=\A_{f^{(1)}(\theta_1)}\,.
\end{equation}
\end{proposition}
\begin{proof}
Let $j_p^1X\in\A_{\theta_1}$. This means that $X^{(0)}_{\theta_0}\in\K_{\theta_1}$, where $\theta_0=\pi_{1,0}(\theta_1)$. From \eqref{TrnsfrmK1}, we get  $f^{(0)}_*(X^{(0)}_{\theta_0})\in\K_{f^{(1)}(\theta_1)}$. From    $f^{(0)}_*(X^{(0)}_{\theta_0})=(f_*X)^{(0)}_{f^{(0)}(\theta_0)}$, we get  $j_{f(p)}^1\bigl(f_*(X)\bigr)\in\A_{f^{(1)}(\theta_1)}$. By the definition $j_p^2f(j_p^1X)=j_{f(p)}^1\bigl(f_*(X)\bigr)$. 
\end{proof}

%-----------------------------------------------------------------------------------------------------------------------
\subsection{Horizontal subspaces}
From definitions \eqref{IztrpAlg0} and \eqref{IstrpSpc} we get 
\begin{equation}\label{IztrpAlgInIztrpSpc}
  g_{\theta_0}\subset\A_{\theta_1}.
\end{equation}
From definition of $\A_{\theta_1}$, we get that
\begin{equation}\label{PrjIztrpSpc}
  \xi_{1,0}(\A_{\theta_1})= T_pM\,.
\end{equation}
In addition, 
$$
  \ker\xi_{1,0}|_{\A_{\theta_1}}=g_{\theta_0}.
$$

We say that a 2n--dimensional subspace $H$ of $\A_{\theta_1}$ is {\it horizontal} if the natural projection
$$
  \xi_{1,0}\bigl|_H:H\longrightarrow T_pM,\quad\xi_{1,0}: j_p^1X\mapsto X_p,
$$
is an isomorphism. From \eqref{EqIstrpSpc}, we obtain that there are horizontal subspaces in the space $\A_{\theta_1}$.  

Let $H$ be a horizontal subspace of $\A_{\theta_1}$, then obviously
$$
  \A_{\theta_1}=H\oplus g_{\theta_0}\,.
$$

Any two horizontal subspaces $H$ and $\tilde H$ of $\A_{\theta_1}$ define the linear map
$$
  f_{H,\tilde H}:T_pM\to g_{\theta_0}\,,\quad
  f_{H,\tilde H}: X\mapsto (\xi_{1,0}|_H)^{-1}(X)
  -(\xi_{1,0}|_{\tilde H})^{-1}(X)\,.
$$
Let $H\subset\A_{\theta_1}$ be a horizontal subspace and let $f:T_pM\to g_{\theta_0}$ be a linear map. Then there exist a unique horizontal subspace $\tilde H\subset\A_{\theta_1}$ such that $f=f_{H,\tilde H}$. This subspace is spanned by the $1$--jets $(\xi_{1,0}|_H)^{-1}(X)-f(X)$, $X\in T_pM$.

Let $f$ be a diffeomorphism of $M$ such that $\theta_1$ belongs to the domain of $f^{(1)}$. Then obviously, we have 
\begin{proposition}\label{TrnsfrmHS} Let $H$ be a horizontal subspace of $\A_{\theta_{k+1}}$. Then $j_p^2f(H)$ is a   horizontal subspace of $\A_{f^{(1)}(\theta_1)}$.
\end{proposition}

%%%%%%%%%%%%%%%%%%%%%%%%%%%%%%%%%%%%%%%%%%%%%%%%%%%%%%%%%%%%%%%%%%%%%%%%%%%%%%%%%%%%%%%%%%%%%%%%%%%%%%%%%%%%%%%%%%%%%%%%
\section{Differential invariants}
In this section, we find a differential invariant of order 1 of the action of $\Gamma$ on $\pi$. For $n=1$, this invariant is trivial; for $n>1$, this invariant is nontrivial. Finally, we compare the founded invariant with the Nijenhuise tensor.  

%-----------------------------------------------------------------------------------------------------------------------
\subsection{The structure function}
Let $p\in M$. Consider the vector space $J_p^1\xi$ of 1--jets at $p$ of all vector fields in $M$ passing through $p$ and the bilinear map
$$
  [\,\cdot\,,\,\cdot\,]:J_p^1\xi\times J_p^1\xi\longrightarrow T_pM\,,\quad
  [j_p^1X, j_p^1Y]=[X, Y]_p\,,
$$ 
where $[X, Y]_p$ is the value at the point $p$ of the bracket of the vector fields $X$ and $Y$.  

Every horizontal subspace $H\subset\A_{\theta_1}$ generates the exterior 2--form $\omega_H$ on $T_pM$ with values in $T_pM$ by the formula 
\begin{equation}\label{OmgH}
  \omega_H(X_p, Y_p)=[\,(\xi_{1,0}\bigl|_H)^{-1}(X_p),\,(\xi_{1,0}\bigl|_H)^{-1}(Y_p)\,]\,,\quad
  \forall\,X_p, Y_p\in T_pM.
\end{equation}
Let $\theta_0=\pi_{1,0}(\theta_1)$. The form $\omega_H$ defines the element 
$$
  \chi(\theta_1)=\omega_H\,+\,\partial_{1,1}(g_{\theta_0}\otimes T_p^*M)\,.
$$ 
of the cohomology goup $H^{0,2}(g_{\theta_0})=T_pM\otimes\wedge^2 T_p^*M/\partial_{1,1}(g_{\theta_0}\otimes T_p^*M)$ (see \eqref{SpncrCmplx}).
\begin{proposition}
 The class $\chi(\theta_1)$ is independent of the choice of a horizontal subspace $H$ in $\A_{\theta_1}$. 
\end{proposition}
\begin{proof} Suppose $H$ and $\tilde H$ are horizontal subspaces of $\A_{\theta_1}$. Then in coordinates, we have
$H=\{\,(X^i,\,h^i_{j,r}X^r)\,\}$ and $\tilde H=\{\,(X^i,\,\tilde h^i_{j,r}X^r)\,\}$. Hence 
$\omega_H((X_p, Y_p))=X^rY^s(h^i_{r,s}-h^i_{s,r})$ and 
$\omega_{\tilde H}((X_p, Y_p))=X^rY^s(\tilde h^i_{r,s}-\tilde h^i_{s,r})$. Taking into account that 
$(h^i_{j,r}-\tilde h^i_{j,r})\in g_{\theta_0}\otimes T_p^*M$, we obtain that
$\omega_H-\omega_{\tilde H}\in \partial_{1,1}(g_{\theta_0}\otimes T_p^*M)$.
\end{proof}

Let $f$ be a point transformation of the base of $\pi$ such that $\theta_1$ belongs to the domain of $f^{(1)}$. From propositions \ref{PrpstnTrnsfrmIstrpSpc} and \ref{TrnsfrmHS}, we obviously get that 
\begin{equation}\label{TrnsfrmOmgH}  
 f_*\bigl(\omega_H(X_p, Y_p)\bigr)=\omega_{j_p^2f(H)}\bigl(f_*(X_p), f_*(Y_p)\bigr)
  \,,\quad\forall\,X_p, Y_p\in T_p\,.
\end{equation}
Hence the element $\chi(\theta_1)$ is defined by $\theta_1$ in the natural way. Thus we obtain
\begin{proposition}
 The field
 $$
   \chi:\theta_1\longmapsto\chi(\theta_1) 
 $$
 on $J^1\pi$ is a 1-st order differential invariant of the action of $\Gamma$ on $\pi$.
\end{proposition}

%-----------------------------------------------------------------------------------------------------------------------
\subsection{The natural complement}
In this section for $n>1$, we find a natural complementation to the space $\partial_{1,1}(g_{\theta_0}\otimes T_p^*M)$ in the space $T_pM\otimes(\wedge^2T_p^*M)$. 

Let $\theta_0$ be a point of $J^0\pi=E\subset TM\otimes T^*M$ and $p=\pi(\theta_0)$. Then $\theta_0$ generates the linear map
$$
  \widehat{\theta}_0 :T_pM\otimes T_p^*M\longrightarrow T_pM\otimes T_p^*M\,,\quad 
  \widehat{\theta}_0:X\mapsto X\theta_0-\theta_0X\,.
$$
Obviously
$$
   \ker \widehat{\theta}_0=g_{\theta_0}\,. 
$$
\begin{lemma}
$$
  \ker \widehat{\theta}_0\cap\im \widehat{\theta}_0=\{0\}
$$ 
\end{lemma}
\begin{proof} Suppose $X\in\ker\widehat{\theta}_0\cap\im \widehat{\theta}_0$. Then $X\theta_0-\theta_0X=0$ and there is $Y\in TM\otimes T^*M$ such that $X=Y\theta_0-\theta_0Y$. It follows $0=(Y\theta_0-\theta_0Y)\theta_0-\theta_0(Y\theta_0-\theta_0Y)=-2(Y+\theta_0Y\theta_0)$. Hence $Y=-\theta_0Y\theta_0$. Therefore $\theta_0Y=Y\theta_0$, that is $Y\in\ker \widehat{\theta}_0$. This means that $X=0$.
\end{proof}
From this lemma, we get the direct sum decomposition
\begin{equation}\label{DrctSmDcmstn}
  T_pM\otimes T_p^*M=g_{\theta_0}\oplus\im\widehat{\theta}_0\,.
\end{equation}
Thus every $X\in T_pM\otimes T_p^*M$ can be uniquely decomposed in following way 
\begin{equation}\label{DrctSmDcmstn0}
  X=\frac{1}{2}(X-\theta_0X\theta_0)+\frac{1}{2}(X+\theta_0X\theta_0)\,,
\end{equation}
where $(1/2)(X-\theta_0X\theta_0)\in g_{\theta_0}$ and $(1/2)(X+\theta_0X\theta_0)\in\im\widehat{\theta}_0$. 
\begin{lemma}
Let $f$ be a local diffeomorphism of $M$ defined in a neighborhood of $p$. Then 
$$
  \im\widehat{f^{(0)}(\theta_0)}=f^{(0)}(\im \widehat{\theta}_0)\quad\text{and}\quad
  \ker\widehat{f^{(0)}(\theta_0)}=f^{(0)}(\ker \widehat{\theta}_0)
$$
\end{lemma}
\begin{proof} We have 
$f^{(0)}\bigl(\widehat{\theta}_0(X)\bigr)=f^{(0)}(X\theta_0)-f^{(0)}(\theta_0X)
  =f^{(0)}(X)f^{(0)}(\theta_0)-f^{(0)}(\theta_0)f^{(0)}(X)=\widehat{f^{(0)}(\theta_0)}\bigl(f^{(0)}(X)\bigr)$.
This proves the first equality. Let $X\in \ker \widehat{\theta}_0$. Then
$\widehat{f^{(0)}(\theta_0)}\bigl(f^{(0)}(X)\bigr)=f^{(0)}(X)f^{(0)}(\theta_0)-f^{(0)}(\theta_0)f^{(0)}(X)
\\=f^{(0)}\bigl(\widehat{\theta_0}(X)\bigr)=0$.
This proves the second equality.
\end{proof}
This lemma means that $\theta_0$ generates direct sum decomposition \eqref{DrctSmDcmstn} in the natural way.

Consider the Spencer complex, see \eqref{SpncrCmplx}, generated by the algebra $T_pM\otimes T_p^*M$
\begin{multline}\label{SpncrCmplx1}
  0\xrightarrow{}T_pM\otimes(\odot^2T_p^*M)\hookrightarrow\\ 
  \hookrightarrow (T_pM\otimes T_p^*M)\otimes T_p^*M
  \xrightarrow{\partial_{1,1}}T_pM\otimes(\wedge^2 T_p^*M)\xrightarrow{} 0\,.
\end{multline}
It is easy to calculate that $\dim T_pM\otimes(\odot^2T_p^*M)=4n^3+2n^2$, $\dim (T_pM\otimes T_p^*M)\otimes T_p^*M=8n^3$, and 
$\dim T_pM\otimes(\wedge^2T_p^*M)=4n^3-2n^2$. It follows that 
\begin{equation}\label{Srjctn}
  \partial_{1,1}\bigl((T_pM\otimes T_p^*M)\otimes T_p^*M \bigr)=T_pM\otimes\wedge^2 T_p^*M\,.
\end{equation}
From \eqref{DrctSmDcmstn}, we get the following natural decomposition 
\begin{equation}\label{DrctSmDcmstn1}
  (T_pM\otimes T_p^*M)\otimes T_p^*M=(g_{\theta_0}\otimes T_p^*M)\oplus(\im\widehat{\theta}_0\otimes T_p^*M)\,.
\end{equation}
\begin{lemma}
 $$
   \partial_{1,1}(g_{\theta_0}\otimes T_p^*M)\cap\partial_{1,1}(\im\widehat{\theta}_0\otimes T_p^*M)=\{0\}
 $$ 
\end{lemma}
\begin{proof} Suppose $Z\in\partial_{1,1}(g_{\theta_0}\otimes T_p^*M)\cap\partial_{1,1}(\im\widehat{\theta}_0\otimes T_p^*M)$. Then there are $X\in g_{\theta_0}\otimes T_p^*M$ and $Y\in\im\widehat{\theta}_0\otimes T_p^*M$ so that $\partial_{1,1}(X)=\partial_{1,1}(Y)=Z$, that is $\partial_{1,1}(X-Y)=0$. From \eqref{SpncrCmplx1}, this means that $X-Y\in T_pM\otimes(\odot^2T_p^*M)$. Taking into account \eqref{DrctSmDcmstn}, we get
$$
  T_pM\otimes(\odot^2T_p^*M)=(T_pM\otimes T_p^*M)^{(1)}=(g_{\theta_0}\oplus\im\widehat{\theta}_0)^{(1)}
  =(g_{\theta_0})^{(1)}\oplus(\im\widehat{\theta}_0)^{(1)}
$$
Hence there are $X_1\in(g_{\theta_0})^{(1)}$ and $Y_1\in (\im\widehat{\theta}_0)^{(1)}$ so that $X-Y=X_1+Y_1$, that is $X-X_1=Y-Y_1$. Taking into account that $X-X_1,Y-Y_1\in (T_pM\otimes T_p^*M)\otimes T_p^*M$, $\im(X-X_1)\subset g_{\theta_0}$, $\im(Y-Y_1)\subset\im\widehat{\theta}_0$, and \eqref{DrctSmDcmstn}, we obtain that $X-X_1=Y-Y_1=0$, that is $X=X_1$ and $Y=Y_1$. It follows $Z=\partial_{1,1}(X_1)=\partial_{1,1}(Y_1)=0$. This completes the prove.
\end{proof}

Now from \eqref{DrctSmDcmstn1}, \eqref{Srjctn}, and this lemma, we get the following natural decomposition
\begin{equation}\label{DrctSmDcmstn2}
  T_pM\otimes(\wedge^2T_p^*M)
  =\partial_{1,1}(g_{\theta_0}\otimes T_p^*M)\oplus\partial_{1,1}(\im\widehat{\theta}_0\otimes T_p^*M)\,.
\end{equation}

%-----------------------------------------------------------------------------------------------------------------------
\subsection{The invariant 2--form}
From \eqref{DrctSmDcmstn2}, we get that for every $\theta_1\in J^1\pi$, there exists a unique representative $\omega_{\theta_1}\in\chi(\theta_1)$ such that 
$$
  \omega_{\theta_1}\in\partial_{1,1}(\im\widehat{\theta}_0\otimes T_p^*M)\,.
$$
It follows from state above
\begin{proposition}
  The field of 2--forms 
 $$
  \omega: \theta_1\longmapsto\omega_{\theta_1}
 $$
 on $J^1\pi$ is a 1-st order differential invariant of the action of $\Gamma$ on $\pi$.
\end{proposition}
 
Let us calculate $\omega$ in standard coordinates. Suppose $\theta_1=(x^q,\,u^i_j,\,u^i_{j,r})$. Then from \eqref{EqIstrpSpc}, we get that a horizontal subspace $H=\{\,(X^i,\,h^i_{j,r}X^r)\,\}$ of $\A_{\theta_1}$ is described by the equation
\begin{equation}\label{HrzntSbspc}
  -u^i_{j,k}+ h^i_{r,k}\,u^r_j-u^i_r\,h^r_{j,k}=0\,.
\end{equation}
We have that $(h^i_{j,r})\in (T_pM\otimes T_p^*M)\otimes T_p^*M$. From \eqref{DrctSmDcmstn0} and \eqref{DrctSmDcmstn1}, we obtain that the natural projection of $(h^i_{j,r})$ on $\im\widehat{\theta}_0\otimes T_p^*M$ along  $g_{\theta_0}\otimes T_p^*M$ is $\frac{1}{2}(h^i_{j,k}+u^i_sh^s_{r,k}u^r_j)$. From \eqref{HrzntSbspc}, we get that 
$\frac{1}{2}(h^i_{j,k}+u^i_sh^s_{r,k}u^r_j)=-\frac{1}{2}u^i_{r,k}u^r_j$. Therefore $\partial_{1,1}\bigl((-\frac{1}{2}u^i_{r,k}u^r_j)\bigr)=\frac{1}{2}(-u^i_{r,k}u^r_j+u^i_{r,j}u^r_k)$. Thus in standard coordinates
\begin{equation}\label{0mg}
  \omega=\frac{1}{2}(u^i_{r,j}u^r_k-u^i_{r,k}u^r_j)\frac{\partial}{\partial x^i}\otimes(dx^j\wedge dx^k)\,.
\end{equation}

Let $S$ be a section of $\pi$. It generates the section $j_1S$ of the bundle $\pi_1$ by the formula 
$$
  j_1S:p\mapsto j_p^1S\,,\quad\forall\,p\in M\,.
$$ 
By $\omega |_{j_1S}$ we denote the restriction of $\omega$ to the image of $j_1S$. This restriction can be considered as an invariant vector--valued 2--form on $M$. From \eqref{0mg}, we get that 
\begin{equation}\label{Omg1}
  (\omega |_{j_1S})^i_{jk}=\frac{1}{2}(S^i_{r,j}\,S^r_k - S^i_{r,k}\,S^r_j)\,,
\end{equation}
where $S^i_{j,k}=\partial S^i_j/\partial x^k$. By $N$ we denote the Nijenhuise tensor of the almost--complex structure $S$. Recall, see \cite{KbshNmz}, that 
\begin{equation}\label{Nnhs}
  N^i_{jk}
  =2(S^i_{k,r}\,S^r_j-S^i_{j,r}\,S^r_k-S^r_{k,j}\,S^i_r+S^r_{j,k}\,S^i_r)\,.
\end{equation} 
\begin{theorem} 
 $$
   N=0\quad\text{iff}\quad\omega|_{j_1S}=0\,.
 $$
\end{theorem}
\begin{proof} Let $p$ be an arbitrary point of the domain of $S$. 

Suppose $N=0$. Then in a neighborhood of $p$, there exist local coordinates $x^1,\,\ldots,\,x^{2n}$ so that $S(x^1,\ldots,x^{2n})=\begin{pmatrix}0&-I\\I& 0\end{pmatrix}$, see \cite{NwlndrNrnbrg}. Now by \eqref{Omg1}, we get that $\omega|_{j_1S}=0$.

Let $\omega|_{j_1S}=0$. Check local coordinates in a a neighborhood of $p$ such that $S(p)=\begin{pmatrix}0&-I\\I& 0\end{pmatrix}$. This means that 
$$
  S^i_j(p)=\left\{
  \begin{aligned}
  -1&\quad\text{if}\quad i+j=2n+1\quad\text{and}\quad i\leq n\,, \\
   1&\quad\text{if}\quad i+j=2n+1\quad\text{and}\quad i\geq n+1\,, \\
   0&\quad\text{if}\quad i+j\neq 2n+1\,.
  \end{aligned}\right.
$$
It follows from \eqref{Nnhs} that
\begin{multline*}
  N^i_{jk}(p)
  =2(S^i_{k,2n+1-j}S^{2n+1-j}_j - S^i_{j,2n+1-k}S^{2n+1-k}_k\\
    -S^{2n+1-i}_{k,j}S^i_{2n+1-i} + S^{2n+1-i}_{j,k}S^i_{2n+1-i})\,,
\end{multline*}
where there is no summation over repeated indexes of the form $2n+1-s$.

Show that the term $2(S^i_{k,2n+1-j}S^{2n+1-j}_j - S^i_{j,2n+1-k}S^{2n+1-k}_k)$ of $N^i_{jk}(p)$ is zero.
From \eqref{Omg1}, we get
\begin{multline*}
  (\omega |_{j_1S})^i_{2n+1-j\,2n+1-k}(p)=\frac{1}{2}(S^i_{k,2n+1-j}S^k_{2n+1-k} 
   - S^i_{j,2n+1-k}S^j_{2n+1-j})\\
   =-\frac{1}{2}(S^i_{k,2n+1-j}S_k^{2n+1-k} - S^i_{j,2n+1-k}S_j^{2n+1-j})
\end{multline*}
We have that either $S_k^{2n+1-k}=S_j^{2n+1-j}$ or $S_k^{2n+1-k}=-S_j^{2n+1-j}$. It follows that the considered term of $N^i_{jk}(p)$ is equal to $\pm4(\omega |_{j_1S})^i_{2n+1-j\,2n+1-k}(p)$. This means that this term of $N^i_{jk}(p)$ is equal to zero. 

Show that the last term $2\bigl((S^{2n+1-i}_{j,k}-S^{2n+1-i}_{k,j})S^i_{2n+1-i}\bigr)$ of $N^i_{jk}(p)$ is zero.
From \eqref{Omg1}, we get
\begin{multline*}
  (\omega |_{j_1S})^{2n+1-i}_{j\,k}(p)=\frac{1}{2}(S^{2n+1-i}_{2n+1-k,j}S_k^{2n+1-k} 
   - S^{2n+1-i}_{2n+1-j,k}S_j^{2n+1-j})\\
   =\frac{1}{2}(S^{2n+1-i}_{r,j}S_k^r - S^{2n+1-i}_{r,k}S_j^r)\,,
\end{multline*}   
where there is summation over the repeated index $r$. Taking into account \eqref{J1pi}, we obtain
\begin{multline*}
   \frac{1}{2}(S^{2n+1-i}_{r,j}S_k^r - S^{2n+1-i}_{r,k}S_j^r)
   =-\frac{1}{2}(S^{2n+1-i}_rS_{k,j}^r - S^{2n+1-i}_rS_{j,k}^r)\\
   =-\frac{1}{2}(S^{2n+1-i}_iS_{k,j}^i - S^{2n+1-i}_iS_{j,k}^i)
   =-\frac{1}{2}(S_{k,j}^i - S_{j,k}^i)S^{2n+1-i}_i
\end{multline*}
Therefore $S_{k,j}^i - S_{j,k}^i=0$ for all $i,j,k$. It follows that the second considered term of $N^i_{jk}(p)$ is equal to zero.
\end{proof}

%%%%%%%%%%%%%%%%%%%%%%%%%%%%%%%%%%%%%%%%%%%%%%%%%%%%%%%%%%%%%%%%%%%%%%%%%%%%%%%%%%%%%%%%%%%%%%%%%%%%%%%%%%%%%%%%%%%%%%%%

\end{document}